\documentclass{amsart}
\usepackage[utf8]{inputenc}
\usepackage{amssymb,latexsym}
\usepackage{amsmath}
\usepackage{graphicx}
\usepackage{textcomp}

\usepackage{amsthm,amssymb,enumerate,graphicx, tikz}
\usepackage{amscd}
\usepackage{setspace}
\usepackage{comment}
\usepackage{hyperref}
\usepackage{cleveref}
\usepackage{subcaption}

\newtheorem{theorem}{Theorem}[section]

\newtheorem*{theorem*}{Theorem}

\newtheorem{claim}[theorem]{Claim}

\theoremstyle{definition}
\newtheorem{definition}[theorem]{Definition}

\theoremstyle{remark}
\newtheorem{remark}[theorem]{Remark}
\newtheorem{example}[theorem]{Example}

\newcommand{\F}{\mathcal{F}}

\newcommand{\M}{\mathcal{M}}

\newcommand{\conv}{\textrm{conv}}
\newcommand{\supp}{\textrm{supp}}
\newcommand{\cl}{\textrm{cl}}

\newcommand{\I}{\mathcal{I}}

\newcommand{\PP}{\mathcal{P}}

\title{Matroid colorings of KKM covers}
\author{Daniel McGinnis}\thanks{D. McGinnis: Department of Mathematics, Princeton University, USA.  \url{dm7932@princeton.edu}. Supported by NSF award no. 2402145.} 

\begin{document}

\begin{abstract}
    We prove a KKM-type theorem for matroid colored families of set coverings of a polytope. This generalizes Gale's colorful KKM theorem as well as recent sparse-colorful variants by Sober\'on, and McGinnis and Zerbib. 
\end{abstract}

\maketitle

\section{Introduction}

The \textit{KKM theorem} by Knaster, Kuratowski, and Mazurkiewicz \cite{knaster1929} is a theorem about set coverings of the simplex:

\begin{theorem}[The KKM Theorem \cite{knaster1929}]
If a family of closed subsets $(A_1,\dots,A_k)$ of the $(k-1)$-dimensional simplex 
$\Delta^{k-1} =\conv\{v_1,\dots,v_k\}$  satisfies $\sigma \subseteq \bigcup_{v_i \in \sigma} A_i$ for 
every face $\sigma$ of $\Delta^{k-1}$ (including for  $\sigma=\Delta^{k-1}$), then $\bigcap_{i=1}^k A_i \neq \emptyset$. 
\end{theorem}

We call a family $(A_1,\dots,A_k)$ satisfying the conditions of the KKM theorem a \textit{KKM cover}. We note that the KKM theorem still holds if all the sets are open.

Notably, the KKM theorem is closely related to the Brouwer's fixed point theorem and Sperner's Lemma \cite{spernerslemma1928} in the sense that these statements can easily be deduced from one another. As with the other two theorems, the KKM theorem has numerous applications in different areas of mathematics. See \cite{mcginnis2024usingkkmtheorem} for a recent survey that discusses various such applications. 

Numerous generalizations and extensions of the KKM theorem have been proven over the past several decades, and the development of such results continue to be explored to this day. Notable such generalizations include Gale's colorful version of the KKM theorem \cite{galeequilibrium1984}, the KKMS theorem due to Shapley \cite{shapleybalanced1973}, a version for set coverings of general polytopes by Komiya \cite{Komiyasimple1994}, and, more recently, a colorful version of Komiya's theorem due to Frick and Zerbib \cite{floriancolorful2019}.

The the aforementioned notion of set covering for general polytopes present in Komiya's theorem generalizes KKM covers and is relevant for the statement of our main result, so we explicitly describe this type of set covering, which we  call a \textit{Komiya cover}, here. For a polytope $P$, we define $F(P)$ to be the set of nonempty faces of $P$.

\begin{definition}
    Let $P$ be a polytope and let $(A_\sigma \mid \sigma \in F(P))$ be a family of closed sets of $P$ such that for every face $\tau$ of $P$, we have that $\tau \subset \bigcup_{\sigma\subset \tau} A_\sigma$. Then $(A_\sigma \mid \sigma \in F(P))$ is called a \textit{Komiya cover} of $P$.
\end{definition}

Note that a KKM cover of a simplex can be regarded as a Komiya cover where only the sets in the cover indexed by vertices are nonempty.

A new and exciting direction along the vein of these KKM type theorems was taken by Sober\'on \cite{soberon2022fair} who proved the following sparse-colorful version of the KKM theorem.

\begin{theorem}[Sober\'on \cite{soberon2022fair}]\label{thm:sparseKKM}
    Let $n\geq k$ and suppose we have $n$ families of closed subsets of $\Delta_{k-1}$, $(A_1^j,\dots,A_k^j)$ for $j\in [n]$, such that for any $I\subset \binom{[n]}{n-k+1}$, the family 
    \[
    \left(\bigcup_{j\in I} A_1^j,\dots,\bigcup_{j\in I} A_k^j\right)
    \]
    is a KKM cover. Then there exists an injective function $\pi: [k] \rightarrow [n]$ such that $\bigcap_{i=1}^k A_i^{\pi(i)} \neq \emptyset$.
\end{theorem}

Theorem \ref{thm:sparseKKM} is colorful in the sense that we can think of each family $(A_1^j,\dots,A_k^j)$ as being colored by the index $j$, and it is sparse in the sense that not all colors are represented in the intersection of sets in the conclusion of the theorem (unless $n=k$). Soon after, a generalization of Theorem \ref{thm:sparseKKM} to set coverings of general polytopes was proven by McGinnis and Zerbib \cite{mcginnis2024sparse}.

When $n=k$, i.e. there are as many families as the dimension of the simplex that is being covered, Theorem \ref{thm:sparseKKM} recovers the aforementioned theorem of Gale. If in addition, we have that all the families are identical, we recover the original KKM theorem.

The main result of this paper is a common generalization of all previously mentioned extensions of the KKM theorem. In this result, we color families of sets by elements of a given \textit{matroid}. We provide some basic matroid theoretic definitions; see \cite{oxley2011} for a comprehensive treatment of matroids. 

\begin{definition}
    A matroid $\M$ consists of a set $W$ ($W$ will always be finite for our purposes), called the ground set, and a collection $\mathcal{I}$ of subsets of $W$ called the independent sets of $\M$ satisfying the following properties:
    \begin{enumerate}
        \item The empty set is independent, i.e. $\emptyset \in \mathcal{I}$,
        \item if $A\in \I$ and $B\subset A$, then $B\in \I$,
        \item if $A,B\in \I$ and $|A| > |B|$, then there exists an element $x\in A\setminus B$ such that $B\cup \{x\} \in \I$.
    \end{enumerate}
\end{definition}
One typical example of a matroid is given by taking a finite set of vectors $V$ from a vector space to be the ground set and taking the independent sets to consist of the sets of vectors that are independent in the usual linear-algebraic sense; such matroids are called realizable. Hence, matroids generalize the usual notion of independence in vector spaces, and in fact, almost all matroids are not realizable \cite{nelson2018nonrepresentable}. 

Given a matroid $\M$ with ground set $W$, the \textit{rank function} $r_\M: \PP(W) \rightarrow \mathbb{Z}_{\geq 0}$ is defined by the following rule: for $A\subset W$, $r_\M(A)$ is the size of the largest independent set contained in $A$. When the matroid $\M$ is understood from context, we may write $r$ instead of $r_\M$ for the rank function. The rank of $\M$ is defined to be $r_\M(W)$, i.e. the size of the largest independent set of $\M$, and such an independent set of largest size is called a \textit{basis}. We will introduce further notions and definitions as we need. The following matroid theoretic notion of a Komiya cover will be relevant to our main theorem. Recall that for a polytope $P$, we denote $F(P)$ to be the set of nonempty faces of $P$.

\begin{definition}[$\M$-Komiya cover]
    Let $P$ be a $(k-1)$-dimensional polytope and let $\M$ be a matroid with rank $k$ and ground set $W$. An $\M$-Komiya cover is a family of closed sets of $P$ 
    \[
    (A_\sigma^w \mid w \in W,\, \sigma \in F(P))
    \]
    such that for all $G\subset W$ with $r(W\setminus G)\leq k-1$,
    \[
    \left(\bigcup_{w\in G} A_\sigma^w \mid \sigma\in F(P) \right)
    \]
    is a Komiya cover of $P$.
\end{definition} 

\begin{theorem}[Main Theorem]\label{thm:main}
    Let $P$ be a $(k-1)$-dimensional polytope with $p\in P$, and let $\M$ be a matroid of rank $k$ on the ground set $W$. For each face $\sigma\in F(P)$ and each element $w\in W$, let $y_\sigma^w\in \sigma$ be a point and let $A^w_\sigma$ be a closed set such that
    \[
    (A_\sigma^w \mid w \in W,\, \sigma \in F(P))
    \]
    is an $\M$-Komiya cover. Then there exists a basis $\{w_1,\dots, w_k\}$ of $\M$ and faces $\sigma_1,\dots,\sigma_k$ such that $\bigcap_{i=1}^k A^{w_i}_{\sigma_i} \neq \emptyset$ and $p\in \conv(\{y_{\sigma_1}^{w_1},\dots,y_{\sigma_k}^{w_k}\})$.
\end{theorem}

\begin{remark}
As is the case for the main results in \cite{floriancolorful2019} and \cite{mcginnis2024sparse} for instance, Theorem \ref{thm:main} is true if all the sets $A_\sigma^w$ are open.
\end{remark}

The KKM theorem is recovered when $P=\Delta_{k-1}$, $\M$ is the matroid of rank 1 consisting of just 1 element, and only the sets indexed by a vertex are nonempty. Komiya's theorem is obtained by again taking $\M$ to be the rank 1 matroid with just 1 element. Theorem \ref{thm:sparseKKM} is obtained by taking $P=\Delta_{k-1}$, $\M$ to be the matroid on ground set $[n]$ whose bases are the $k$-element subsets of $[n]$, and only the sets indexed by a vertex are nonempty.

Another way to deduce Gale's colorful KKM theorem is through a partition matroid. Let $(A_1^j,\dots,A_k^j)$ be closed sets that are KKM covers for each $1\leq j\leq k$. Consider the partition matroid with vertex set the ordered pairs $(i,j)$ where $1\leq i,j \leq k$ and with the $k$ parts $\{(1,j),\dots,(k,j)\}$ for $1\leq j\leq k$. We then take $A^{(i,j)}_i=A_i^j$ and $A^{(i,j)}_\ell = \emptyset$ for $\ell \neq i$. We then apply Theorem \ref{thm:main} to the families of sets
\[
(A^{(i,j)}_1,\dots,A^{(i,j)}_k)_{1\leq i,j \leq k}.
\]
Theorem \ref{thm:sparseKKM} can also be deduced in a similar fashion.

\section{Applications in discrete geometry}

The statement of Theorem \ref{thm:main} bears resemblance to the following generalization of the colorful version of Carath\'eodory's theorem of  B\'ar\'any \cite{barany1982generalization}. This result was observed by Kalai and Meshulam and follows from their topological colorful Helly theorem \cite{kalai2005topologicalhelly}. 

\begin{theorem}\label{thm:matroidColCaratheodory}
    Let $V$ be a finite set of points in $\mathbb{R}^d$, and let $\M$ be a matroid with ground set $V$. If $0\in \conv(G)$ for every $G\subset V$ with $r(V\setminus G) \leq d$, then there exists an independent set $T$ of $\M$ such that $0\in \conv(T)$. 
\end{theorem}

We can think of the point set $V$ as being colored by the matroid $\M$, and that Theorem \ref{thm:matroidColCaratheodory} provides a condition guaranteeing the existence a set of points that both contain the origin in its convex hull and is colorful in the sense that the set is independent in $\M$. To avoid potential confusion, $\M$ is not necessarily the realizable matroid corresponding to $V$, it is simply any matroid with $V$ as the ground set. 

In fact, Theorem \ref{thm:main} can be used to prove Theorem \ref{thm:matroidColCaratheodory} and hence it is a proper generalization of the above result. The proof is almost identical to arguments from \cite{floriancolorful2019}, so we provide a sketch of the proof below that points out the key differences.

\begin{proof}[Proof sketch of Theorem \ref{thm:matroidColCaratheodory}]
    We may assume that $\M$ is \textit{loopless} without loss of generality, i.e., there is no element $x$ where $\{x\}$ is dependent in $\M$. We may then also assume $0\notin V$, otherwise we are done since $\{0\}$ is an independent set. 

    We proceed as in the proof of Theorem 3.1 in \cite{floriancolorful2019}. Let $P \subset \mathbb{R}^d$ be a polytope containing $0$ in its interior such that any two points $x$ and $y$ belonging to the same face of $P$ satisfies $\langle x, y \rangle \geq 0$. For a point $v\in V$, let $\vec{v}$ be the ray through $v$ emanating from the origin. For each nonempty face $\sigma$ of $P$ and $v\in V$, we define $A_\sigma^v$ in the following way. We take $A_P^v = \emptyset$ for all $v\in V$, so assume that $\sigma$ is a proper face of $P$. If $\sigma \cap \vec{v} \neq \emptyset$, then we set $A_\sigma^v = \{x\in P \mid \langle x,v \rangle \geq 0\}$ and $y_\sigma^v = \sigma \cap \vec{v}$. Otherwise, we set $A_\sigma^v = \sigma$ and $y_\sigma^v$ to be any point in $\sigma$.

    It follows by the same arguments as in \cite{floriancolorful2019} that these sets satisfy the conditions of Theorem \ref{thm:main} and hence there exist faces $\sigma_1,\dots,\sigma_{d+1}$ and an independent set $\{v_1,\dots,v_{d+1}\}$ of $\M$ such that $\bigcap_{i=1}^{d+1} A_{\sigma_i}^{v_i} \neq \emptyset$ and $0\in \conv\{y_{\sigma_1}^{v_1},\dots,y_{\sigma_{d+1}}^{d+1}\}$. Again following \cite{floriancolorful2019}, the faces $\sigma_1,\dots,\sigma_{d+1}$ must satisfy $\sigma_i \cap \vec{v_i} \neq \emptyset$, and hence, the fact that $0\in \conv\{y_{\sigma_1}^{v_1},\dots,y_{\sigma_{d+1}}^{d+1}\}$ gives us the desired result.
\end{proof}

    A beautiful result due to Holmsen \cite{holmsen2016intersection} is a purely combinatorial generalization of Theorem \ref{thm:matroidColCaratheodory} in which the point set $V\subset \mathbb{R}^d$ is replaced by an oriented matroid. Additionally, a consequence of this generalization is that the conclusion of Theorem \ref{thm:matroidColCaratheodory} still holds even with the weaker condition that $0\in \conv(G)$ for every $G\subset V$ with $r(V\setminus G) \leq d-1$ provided that $r(\M) > d$. This fact was discovered in a weaker form in \cite{holmsen2008surrounding, arocha2009very}.

    The above proof shows that Theorem \ref{thm:main} is another, separate generalization of Theorem \ref{thm:matroidColCaratheodory}. Furthermore, Theorem \ref{thm:main} can be applied to show that a similar phenomenon occurs for matroid colorings in different contexts, for instance, piercing \textit{$d$-intervals}. A $d$-interval is simply the union of $d$ compact intervals on $\mathbb{R}$, and a \textit{separated} $d$-interval is the union of $d$ compact intervals $h_1,\dots,h_d$ such that $h_i \subset (i,i+1)$ for all $1\leq i\leq d$. The \textit{matching number} of a finite family of $d$-intervals $\F$, denoted $\nu(\F)$, is the maximum size of collection of pairwise disjoint $d$-intervals (also called a \textit{matching}) in $\F$. The \textit{covering number} (or \textit{piercing number}) of $\F$, denoted $\tau(\F)$, is the minimum size of a point set $S\subset \mathbb{R}$ that intersects each member of $\F$. It is easy to show that that $\tau(\F) = \nu(\F)$ when $\F$ is a family of $1$-intervals. A problem that has received a significant amount of attention is that of bounding $\tau(\F)$ is terms of $\nu(\F)$ for families of $d$-intervals when $d\geq 2$. The best upper bound is due to Tardos when $d=2$ \cite{tardos1995transversals} and Kaiser for all $d$ \cite{kaiser1997d-intervals}. These proofs and all other known proofs of the following theorem are topological.

    \begin{theorem}[Tardos \cite{tardos1995transversals}, Kaiser \cite{kaiser1997d-intervals}]\label{thm:d-intervals}
        Let $\F$ be a finite family of $d$-intervals. Then $\tau(\F) \leq (d^2-d+1)\nu(\F)$. Moreover, if $\F$ is a finite family of separated $d$-intervals, then $\tau(\F) \leq (d^2-d)\nu(\F)$.
    \end{theorem}

    Matou\v sek \cite{matousek2001Lower} showed that the bound in Theorem \ref{thm:d-intervals} is at least nearly asymptotically tight by exhibiting families $\F$ of $d$-intervals for which $\tau(\F)/\nu(\F) = \Omega(d^2/\log^2 d)$.

    More recently, Frick and Zerbib proved the following colorful generalization of Theorem \ref{thm:d-intervals}.

    \begin{theorem}[Frick and Zerbib \cite{floriancolorful2019}]\label{thm:Cold-intervals}
        For $i\in [k]$, let $\F_i$ be a finite family of $d$-intervals and let $\F= \bigcup_{i\in [k]} \F_i$.
        \begin{enumerate}
            \item[(1)] If $\tau(\F_i) > k-1$ for all $i$, then there exists a matching $M$ of $\F$ such that $|M\cap \F_i|\leq 1$ for all $i$ and $|M| \geq \frac{k}{d^2-d+1}$.
            \item[(2)] If each $\F_i$ is a family of separated $d$-intervals and $\tau(\F_i) > (k-1)d$ for all $i$, then there exists a matching $M$ of $\F$ such that $|M\cap \F_i|\leq 1$ for all $i$ and $|M| \geq \frac{k}{d-1}$.
        \end{enumerate}
    \end{theorem}

Theorem \ref{thm:d-intervals} can be derived from Theorem \ref{thm:Cold-intervals} in the following way. Let $k$ be the smallest value for which $\tau(\F) > k-1$ (or $\tau(\F) > (k-1)d$ if $\F$ is a family of separated $d$-intervals), then take $k$ copies of $\F$ and apply Theorem \ref{thm:Cold-intervals}. A sparse-colorful version of Theorem \ref{thm:Cold-intervals} was proven in \cite{mcginnis2024sparse}. Theorem \ref{thm:main} can be applied to prove a generalization of these $d$-interval results for matroid colorings of families of $d$-intervals.

\begin{theorem}\label{thm:matroid d-intervals}
    Let $\F$ be a finite family of $d$-intervals, and let $\M$ be a matroid with ground set $\F$ and rank $k$.  
    \begin{enumerate}
        \item[(1)] If $\tau(\F') > k-1$ for all subfamilies $\F'\subset \F$ with $r(\F\setminus \F') \leq k-1$, then there exists a matching $M$ of $\F$ that is independent in $\M$ and $|M| \geq \frac{k}{d^2-d+1}$.

        \item[(2)] If $\F$ is a family of separated $d$-intervals and $\tau(\F') > (k-1)d $ for all subfamilies $\F'\subset \F$ with $r(\F\setminus \F') \leq k-1$, then there exists a matching $M$ of $\F$ that is independent in $\M$ and $|M| \geq \frac{k}{d-1}$.
    \end{enumerate}
\end{theorem}
\begin{proof}
    This can proven using Theorem \ref{thm:main} by a straightforward modification of the proof of Theorem 1.3 in \cite{floriancolorful2019} or Theorem 1.8 in \cite{mcginnis2024sparse} in a similar way that the proof of Theorem \ref{thm:matroidColCaratheodory} is a straightforward adaption of Theorem 3.1 in \cite{floriancolorful2019}. Hence, we defer the proof to these existing arguments.
\end{proof}
Theorem \ref{thm:Cold-intervals} can be seen to follow from Theorem \ref{thm:matroid d-intervals} by associating the families of $d$-intervals $\F_1,\dots,\F_k$ with the partition matroid with the $\F_i$'s as the parts.

Another geometric application of Theorem \ref{thm:main} is for line piercing problems for convex sets in the plane. Given a family $\F$ of convex sets in $\mathbb{R}^2$, we define the \textit{line piercing number} of $\F$ to be the minimum number $n$ such that there are $n$ lines whose union intersects each set in $\F$. We say that $\F$ has the \textit{$T(k)$ property} if every $k$ sets of $\F$ can be pierced by a line. The following theorem by Eckhoff \cite{eckhoff1969} states that if $\F$ has the $T(4)$ property, then $\F$ has line piercing number 2, which is best possible.

\begin{theorem}[Eckhoff \cite{eckhoff1969}]\label{thm:T4}
    If $\F$ is a finite family of convex sets in $\mathbb{R}^2$ with the $T(4)$ property, then $\F$ has line piercing number at most 2.    
\end{theorem}

Eckhoff conjectured that such families $\F$ with the $T(3)$ property have piercing number at most 3 \cite{eckhoff1993gallai}, which would be best possible \cite{eckhofftransversal1973}. This has recently been resolved by McGinnis and Zerbib \cite{mcginnis2022line} using a novel application of the KKM theorem developed in \cite{mcginnis43}.

\begin{theorem}[McGinnis and Zerbib \cite{mcginnis2022line}]\label{thm:T3}
    If $\F$ is a finite family of convex sets in $\mathbb{R}^2$ with the $T(3)$ property, then $\F$ has line piercing number at most 3.
\end{theorem}

This same method can be used to provide another proof of Theorem \ref{thm:T4}. Using Theorem \ref{thm:main}, we have a matroid coloring version of Theorem \ref{thm:T4}.

\begin{theorem}\label{thm:MatroidT4}
    Let $\F$ be a finite family of convex sets in $\mathbb{R}^2$, and let $\M$ be a matroid with ground set $\F$ and rank 4. If every 4 sets of $\F$ with rank 4 can be pierced by a line, then there are 2 lines that pierce a subfamily $\F' \subset \F$ such that $r(\F\setminus \F') \leq 3$.
\end{theorem}
\begin{proof}
    The proof is a straightforward adaptation of the proof method in \cite{mcginnis2022line}.
\end{proof}

Corresponding to the piercing result for families with the $T(3)$ property, we have the following matroid coloring version of Theorem \ref{thm:MatroidT3}. In this statement the matroid has rank 6, while the natural statement would for matroids of rank 3. Thus, it is an open problem to determine if the matroid of rank 6 below can be replaced by a matroid of rank 3.

\begin{theorem}\label{thm:MatroidT3}
     Let $\F$ be a finite family of convex sets in $\mathbb{R}^2$, and let $\M$ be a matroid with ground set $\F$ and rank 6. If every 3 sets of $\F$ with rank 3 has can be pierced by a line, then there are 3 lines that pierce a subfamily $\F' \subset \F$ such that $r(\F\setminus \F') \leq 5$.
\end{theorem}
\begin{proof}
    Again, this follows from a straightforward adaptation of the proof method in \cite{mcginnis2022line}.
\end{proof}

\section{Applications in fair division}

Fair division problems typically ask if a given resource can be ``evenly'' distributed among a group of participants. The resource in question and what it means to evenly distribute can vary depending on the situation. We will be primarily concerned with \textit{envy-free} divisions of a ``cake'' which we associate with the interval $[0,1]$. A partition of $[0,1]$ into $n$ subintervals, and an allocation of the subintervals to $n$ guests is said to be envy-free if every guest prefers their piece at least as much as any other piece. In other words, each guest is not envious of the piece that any other guest received. Such a division and allocation of the cake is called an \textit{envy-free division}. The classical envy-free division theorem proves that such a division exists under mild assumptions on the guests. For the statement below, ``piece $i$'' in  a partition of $[0,1]$ into $n$ subintervals is simply $i$'th subinterval that appears from left to right.

\begin{theorem}[Fair division theorem, Stromquist \cite{stromquisthow1980},  Woodall \cite{woodalldividing1980}, 1980]\label{thm:cfd}
Suppose we have a cake and $n$ guests that each satisfy the following conditions. 
\begin{enumerate}
\item[(1)] {\em The players are hungry}: in every partition  of the cake into $n$ pieces every player prefers at least one positive-length piece.
\item[(2)] {\em The preference sets are closed}: if a player prefers piece $i$ in a converging sequence of partitions, then they prefer piece $i$ also in the limit partition.
\end{enumerate}
Then there exists an envy-free division of the cake. 
\end{theorem}

It is well-known that Theorem \ref{thm:cfd} is the colorful KKM theorem in disguise (see \cite{mcginnis2024usingkkmtheorem}). Analogously, Theorem \ref{thm:sparseKKM} can be used to prove a sparse version of Theorem \ref{thm:cfd} as shown in \cite{soberon2022fair}. Along the same lines, we can use Theorem \ref{thm:main} to prove a matroid colorful version. Before stating this result, we will need the following definition, which is an analogue to hungry player condition in Theorem \ref{thm:cfd}.

\begin{definition}
    Suppose we have a group of $n$ players labeled $1$ to $n$, and let $\M$ be a matroid on the ground set $[n]$ with rank $k\leq n$. We say the players are \textit{$\M$-hungry} if for every $G\subset [n]$ with $r([n]\setminus G)\leq k-1$ and any partition of $[0,1]$ into $k$ subintervals, some guest from $G$ prefers a positive-length piece in the partition.
\end{definition}

\begin{theorem}\label{thm:MatroidCake}
    Assume that we have $n$ guests that we label $1$ to $n$, and let $\M$ be a matroid on ground set $[n]$ with rank $k\leq n$. Furthermore, suppose the guests are $\M$-hungry and have closed preference sets. Then there are $k$ guests forming a basis in $\M$ for which there exists an envy-free division of the cake for those players.
\end{theorem}
\begin{proof}
    The proof is a straightforward adaptation of known arguments. See \cite{mcginnis2024usingkkmtheorem} for instance.
\end{proof}

We can also prove an envy-free division result for multiple cakes. Assume now that we have $d$ cakes and we have a partition of each cake into $m$ pieces. A \textit{$d$-tuple of pieces} is simply a choice of one piece from each cake.

\begin{theorem}
    Let $n,m,d$ be positive integers with $d\geq 2$. Suppose we have $n$ guests labeled $1$ to $n$, and let $\M$ be a matroid with ground set $[n]$ and rank $d(m-1)+1$. Furthermore, assume that the preference set of each player is closed and that the following hungry condition is satisfied: For every $G\subset [n]$ with $r([n]\setminus G)\leq d(m-1)$ and for every division of $d$ cakes into $m$ pieces each, there is some guest of $G$ that prefers a non-empty $d$-tuple of pieces. Then there exists $\frac{m}{d-1}$ guests that are independent in $\M$ and an envy-free allocation of pairwise disjoint $d$-tuples of pieces to these players.
\end{theorem}
\begin{proof}
    This is a straightforward adaptation of Theorem 1.12 in \cite{mcginnis2024sparse}.
\end{proof}

\begin{remark}
There is a stronger version of Theorem \ref{thm:cfd} in which the preference set of one of the guests is secrete, which was proven by Woodall \cite{woodalldividing1980}. Another, simpler proof was recently given in \cite{asada2018fair}, and a sparse version was proven in \cite{soberon2022fair}. We leave it as an open problem to formulate and prove a secretive envy-free division theorem for matroid colorings.
\end{remark}

\subsection{Envy-free division in another context}

Theorem \ref{thm:main} allows us to show that there are envy-free divisions exists in different scenarios. In this section, we will consider the following setup. Assume that there are $k$ participants, that we call ``player $i$'' for $1\leq i\leq k$, that have been invited to a cake sharing party. There is also a corresponding graph $G$ with vertex set $[k]$ and edge set $E$, where the existence of an edge $e\in E$ means that the two players corresponding to the endpoints of this edge are willing to share a piece of cake with each other. At this party, the host will bake a cake, slice it into $k-1$ pieces, and serve them to pairs of players that correspond to edges in $G$ (we can think of each piece being assigned to an edge of $G$). Additionally, in order to have a good, inclusive party, the host wants the set of edges that were assigned a piece to form a connected graph (which must be a spanning tree of $G$). Finally, we would like to find an assignment of pieces to edges that is envy-free in the sense that the pair of players corresponding to an edge that was assigned a piece do not prefer another pair's piece over their own. We summarize these notions in the definition below.

\begin{definition}\label{def:edge-payment}
    Let $G$ be a graph with vertex set $[k]$ and assume there are $k$ players corresponding to the vertices of $G$ as in the above discussion. A  \textit{piece-edge allocation} is an allocation of $k-1$ intervals $I_1,\dots,I_{k-1}$, that form a partition of $[0,1]$ (the cake), to edges $e_1,\dots,e_{k-1}$ of $G$ respectively, such that these edges make up a spanning tree of $G$. The piece-edge allocation is said to be \textit{envy-free} if no pair of players corresponding to an edge $e_i$ prefers another pair's piece over their own.
\end{definition}

Theorem \ref{thm:ef-cake-sharing} below provides some natural sufficient conditions for such an envy-free allocation to exist. The first two conditions are analogous to those in Theorem \ref{thm:cfd}. However, in this scenario, it's possible for a pair of players not to prefer any piece in a given partition of the cake. This is why condition (1) in Theorem \ref{thm:ef-cake-sharing} is slightly different from condition (1) in \ref{thm:cfd}. Condition (3) is natural given that we desire the set of edges allocated a piece to form a connected graph (spanning tree). For this condition, we will also need the notion of a \textit{edge cut} of a graph. A set of edges $X\subset E$ is said to be an edge cut of $G$ if the removal of the edges in $X$ yields a disconnected graph.

\begin{theorem}\label{thm:ef-cake-sharing}
    Suppose there are $k$ players and a connected graph $G$ with vertex set $[k]$ and edge set $E$. Suppose that the following conditions are satisfied for pairs of players that correspond to an edge in $G$:
    \begin{enumerate}
        \item[(1)] {\em A hungry-type condition:} if a pair of players prefers some piece in a partition of the cake, this piece will be nonempty.

        \item[(2)] {\em The preference sets are closed:} if a pair of players prefers piece $i$ in a converging sequence of partitions, then they prefer piece $i$ also in the limit partition.

        \item[(3)] {\em Good party:} for any edge cut $X\subset E$ and any partition of the cake into $k-1$ pieces, there is some pair of players corresponding to an edge of $X$ that prefer one of the pieces.
    \end{enumerate}

    Then there exists an envy-free piece-edge allocation.
\end{theorem}
\begin{proof}
    A point $(x_1,\dots,x_{k-1})\in \Delta_{k-2}$ corresponds to a partition of $[0,1]$ into the $k-1$ subintervals $I_j(x) = \left[\sum_{i=1}^{j-1} x_i, \sum_{i=1}^{j} x_i\right)$ for $1\leq j\leq k-1$. We will define a set cover of $\Delta_{k-2}$ with this correspondence in mind with the goal of applying Theorem \ref{thm:main}. The matroid $\M$ in this case will be the \textit{graphic matroid} of rank $k-1$ corresponding to $G$. That is, the matroid with ground set $E$ and whose bases are the sets of edges that form a spanning tree of $G$. Notice that if $X\subset E$ satisfies that the rank of $E\setminus X$ in $\M$ is at most $k-2$, then $X$ is an edge cut of $G$. 

    For $e\in E$ and $1\leq i\leq k-1$, define $A^{e}_i$ to be the set of points $x\in \Delta_{k-2}$ such that the pair of players corresponding to $e$ prefer the piece $I_i(x)$ at least as much as $I_j(x)$ for $1\leq j\leq k-1$ (it is possible that they prefer multiple pieces equally). Conditions (1) - (3) imply that
    \[
    (A^e_i \mid e\in E,\,1\leq i\leq k-1) 
    \]
    is an $\M$-Komiya cover of $\Delta_{k-2}$. Therefore, applying Theorem \ref{thm:main} when $P = \Delta_{k-2}$ and $p$ is for instance the barycenter of $\Delta_{k-2}$, we have that there are edges $e_1,\dots,e_{k-1}$ such that $\bigcap A^{e_i}_i \neq \emptyset$. This means that for any $x\in \bigcap_{i=1}^{k-1} A^{e_i}_i$, the allocation of $I_1(x),\dots,I_{k-1}(x)$ to $e_1,\dots,e_{k-1}$, respectively, is an envy-free piece-edge allocation, which completes the proof.
\end{proof}

\section{Proof of Theorem \ref{thm:main}}

To prove Theorem \ref{thm:main}, we will primarily deal with triangulations of the polytope $P$ with labelings on the vertices, reminiscent of the labeling condition in Sperner's lemma. Specifically, we will work with \textit{Sperner-Shapley} labelings defined below. For a point $v\in P$, $\supp(v)$ is the minimal face of $P$ containing $v$. Recall that $F(P)$ is the set of non-empty faces of $P$.

\begin{definition}
    A \textit{Sperner-Shapley labeling} of a triangulation $T$ of a polytope $P$ is a map
    \[
    \lambda:V(T) \rightarrow F(P)
    \]
    such that $\lambda(v) \subset \supp(v)$ for all $v\in V(T)$.
\end{definition}
We have the following result for Sperner-Shapley labelings, proven in \cite{floriancolorful2019}, that is crucial to our proof.

\begin{theorem}\label{sperner-shapley}
Let $P\subset \mathbb{R}^d$  be a polytope with $p\in P$, and let $T$ be a triangulation of $P$. Let $\lambda: V(T) \rightarrow F(P)$ be a Sperner-Shapley labeling of $T$. Suppose that for every $v\in V(T)$, a point $y(v)\in \lambda(v)$ is assigned.
Then there is a face $\tau$ of $T$ such that $p\in  \conv\{y(v) \mid  v\in V(\tau)\}$.
\end{theorem}

We first prove the following result about refining a triangulation of $P$ to obtain another triangulation with desirable labeling properties. We will show later that Theorem \ref{thm:main} follows from Theorem \ref{sperner-shapley} and Theorem \ref{goodtriangulation} below by a simple limiting argument. For the proof of Theorem \ref{goodtriangulation} below, we will need to recall a few matroid theoretic definitions and properties. A \textit{circuit} of a matroid $\M$ is a minimal dependent set, i.e., every subset of a circuit is an independent set. The \textit{closure} of a set $A$ is defined as
\[
\textrm{cl}(A) = \{x \mid \textrm{rank}(A) = \textrm{rank}(A\cup \{x\})\}.
\]
Note that rank$(A) = \textrm{rank}(\textrm{cl}(A))$.
Finally, recall that the rank function of a matroid is a \textit{submodular function}, meaning that for two subsets $A$ and $B$ of the ground set, we have that
\[
r(A\cup B) + r(A\cap B) \leq r(A) + r(B).
\]

\begin{theorem}\label{goodtriangulation}
Let $k\ge 2$ be an integer. Let $P$ be a $(k-1)$-dimensional polytope with $p\in P$, and let $\M$ be a matroid of rank $k$ on the ground set $W$ with rank function $r$. Assume that for every  $\tau\in F(P)$, we are given points $\{y_\tau^w \mid w \in W\}$, each contained in $\tau$, and we are given an $\M$-Komiya cover  $(A^w_\tau \mid w\in W, \tau \in F(P))$ of $P$. Let $T$ be a convex triangulation of $P$. 
Then there exists a refinement $T'$ of $T$ and maps $\lambda: V(T')\to F(P)$, $f:V(T')\to W$ and $y:V(T') \to P$ with the following properties: 
\begin{enumerate}
    \item[(P1)] $v\in A^{f(v)}_{\lambda(v)}$ for every $v\in V(T')$, 
    \item[(P2)] $y(v) = y_{\lambda(v)}^{f(v)} \in \lambda(v) \subset \supp(v)$ for every $v\in V(T')$,  and
    \item[(P3)] for every face $\tau$ of $T'$, the elements $\{f(v) : v\in \tau\}$ form a basis of $\M$.  
\end{enumerate}
In particular, $\lambda$ is a Sperner-Shapley labeling.
\end{theorem}

\begin{proof}
The general idea behind Theorem \ref{goodtriangulation} is that if we find a triangulation and maps $\lambda, f,y$ satisfying (P1) - (P3), then by Theorem \ref{sperner-shapley}, we can find a full dimensional face $F$ such that $y(V(F))$ contains $p$ in its convex hull. By (P3), we also have that $f(V(F))$ is a basis of $\M$. Then by taking a sequence of finer triangulations, we obtain a sequence of such full dimensional simplices that converges to a point, and this point will satisfy the conclusion of Theorem \ref{thm:main}

    Let $v\in V(T)$. We may choose an element $w \in W$ and a face $\tau \subset \supp(v)$ such that $v\in A^w_\tau$. Let $\lambda(v)=\tau$, $f(v)=w$ and  $y(v)=y_\tau^w$.

    Note that the maps $f,\,y$, and $\lambda$, already satisfy (P1) and (P2). The goal will be to find a refining triangulation $T'$ of $T$ and extensions of the maps $f,\,y$, and $\lambda$ to the new vertices of $T'$ that will satisfy all properties (P1)-(P3).

    To this end, let $B(T)$ be the set of faces $F$ of $T$ such that $\{f(v)\}_{v\in V(F)}$ is a circuit in $\M$. We say that such a face $F$ is \textit{bad} and $B(T)$ is the set of bad faces. In the case that $\{v_1,\, v_2\}$ is an edge of $T$ and $f(v_1) = f(v_2)$, then $\{v_1,\, v_2\}$ is a bad face. Otherwise, a face with more than 2 vertices where some 2 vertices have the same label under $f$ is not considered bad. Let $F_1\in B(T)$ be a bad face of $T$. We will define an algorithm that takes as input $T$, $f$, $y$, and $\lambda$ and returns a refinement $T'$ of $T$ and assignments $f(v)$, $y(v)$, $\lambda(v)$ for each $v\in V(T')\setminus V(T)$ such that the following holds:
    \begin{itemize}
    \item[(i)] properties (P1) and (P2) hold for $T',\lambda,f,y$, 
    \item[(ii)]  $B(T') \subset B(T)$, and
    \item[(iii)] $F_1$ is not a face of $T'$.
\end{itemize}

In particular $|B(T')| < |B(T)|$. Thus, by applying the algorithm at most $|B(T)|$ times, we will obtain a refinement of $T$ satisfying (P1) - (P3).
Throughout the algorithm below, $T_c$, the current triangulation, is repeatedly updated with finer and finer triangulations until we reach a triangulation with the desired properties.

To describe the algorithm, we will need the following notions. Let $F$ be a face in $T$ with vertex set $\{v_1,\dots,v_i\}$. Let $b_F$ be the barycenter of $F$. We will denote $T(F)$ to be the finer triangulation obtained by adding the vertex $b_F$ and replacing every face of $T$ of the form $\{v_1,\dots,v_i,\dots,v_j\}$ with the $i$ faces $\{b_F,v_1,\dots,\hat{v}_r,\dots,v_i,\dots,v_j\}$ for $1\leq r\leq i$ where $\hat{v}_r$ means $v_r$ is not contained in the set. For a triangulation $T$ with labels $\lambda,f,y$ on its vertices and vertex $v$ of $T$, we define
\[
B(T,v) = \{ F \in B(T) \mid v\in F\}.
\]

{\bf The algorithm. }

Setup:
Set $T_c=T(F_1)$. Let $G_{F_1}= \cl(F_1)$, and note that $r(G_{F_1})\leq k-1$. Therefore, there exists $w \in W\setminus G_{F_1}$ such that $b_{F_1}\in A^w_\tau$ for some face $\tau \subset \supp(b_{F_1})$.
Set $f(b_{F_1})=w$, $\lambda(b_{F_1})=\tau$,  and $y(b_{F_1})=y_\tau^w$. Since $y_\tau^w \in \tau \subset \supp(b_{F_1})$, properties (P1) and (P2) hold for $b_{F_1}$.
Set $H=V(F_1)\cup \{b_{F_1}\}$, and initiate a sequence of vertices $S=({F_1})$ (more faces will be added to the sequence throughout the algorithm). For $j\geq 1$, set
\[
Q_{|F_1|+i}=\{F\in B(T_c;b_{F_1}) \mid |V(F)\setminus H| = i \}.
\]

Apply the following procedure:
\begin{enumerate}
    \item If $Q_j=\emptyset$ for all $j\geq 1$,  stop and return $T_c, f, \lambda,  y$. Otherwise, 
    \item Let $j$ be the largest index for which $Q_j\neq \emptyset$, and choose a face $F\in Q_j$. Set $T_c=T_c(b_F)$, and append $F$ to $S$ on the right. Remove $F$ from $Q_j$. 
  \item Let $H = \bigcup_{F'\in S} \left(V(F') \cup \{b_{F'}\} \right)$, and let $G = \cl(f(H\setminus \{b_{F}\}))$. Choose 
    $w\in W\setminus G$ and $
    \tau \subset \supp(b_F)$ such that  $b_F\in A^w_\tau$ (we will show that such a choice exists).\\
 Set $\lambda(b_F)=\tau$, $f(b_F)=w$,  $y({b_F})=y_{\tau}^w$. Like before, properties (P1) and (P2) hold for $v=b_F$.
    \item If $B(T_c; b_F) \neq \emptyset$, set $Q_{j+i}=\{ F' \in B(T_c;b_F) \mid |V(F') \setminus H| = i \}.$ If $B(T_c; b_F) = \emptyset$, remove $F$ from $S$.
\item Return to Step (1).
\end{enumerate}
See Example \ref{exalgorithm} for a demonstration of the steps in the algorithm for a particular triangulation.

\begin{claim}\label{claim:contains}
    Let $S=({F_1},\dots,{F_i})$ be the sequence obtained after performing step (2) in the algorithm, and let $T_c$ be the triangulation obtained after step (2). Assume that $F_i$, when $i\geq 2$, was chosen from the set $Q_j$. Then each maximal simplex of $T_c$ containing the vertex $b_{F_i}$ contains $j-1$ other vertices from $H = \bigcup_{F\in S} \left(V(F) \cup \{b_{F}\} \right)$.
\end{claim}
\begin{proof}

    We will proceed by induction with the base case being when $i=2$. Note that each maximal face containing $b_{F_1}$ contains $|F_1|-1$ vertices from $V(F_1)$. By definition, the triangulation $T_c$ obtained after step (2) is constructed by replacing every face containing $F_2$ with $|F_2|$ new faces by replacing $V(F_2)$ with $b_{F_2}$ and $|F_2| -1$ vertices from $V(F_2)$. Let $T_c'$ be the triangulation for which $T_c=T_c'(F_2)$, i.e. the triangulation occurring immediately before $T_c$ in the algorithm. Since $b_{F_1} \in F_2$, each maximal face in $T_c'$ containing $F_2$ must also contain $|F_1|-1$ vertices from $V(F_1)$. Now recall that $F_2$ contains $j-|F_1|$ vertices outside of $H'=V(F_1)\cup \{b_{F_1}\}$ by definition of $Q_j$. Thus, each maximal face of $T_c$ containing $b_{F_2}$ either contains $b_{F_1}$ or does not contain $b_{F_1}$. If it does contain $b_{F_1}$, then it contains $b_{F_1}$, $|F_1| -1$ vertices from $V(F_1)$, and $|F_2|-1$ vertices from $V(F_2)$, $j-|F_1|-1$ of which lie outside $H'$. If it does not contain $b_{F_1}$, then it contains $|F_1| -1$ vertices from $V(F_1)$, and $|F_2|-1$ vertices from $V(F_2)$, $j-|F_1|$ of which lie outside $H'$. In both cases, each maximal simplex of $T_c$ that contains $b_{F_2}$ contains $j-1$ vertices from $V(F_1)\cup V(F_2) \cup \{b_{F_1},b_{F_2}\}$.

    The inductive step follows similarly. Assume the statement is true for $i-1$. Let $S'=({F_1},\dots, {F_{i-1}})$, $H'=\bigcup_{F\in S'} \left(V(F) \cup \{b_{F}\} \right)$, and $j'$ be the index for which $F_{i-1}$ was chosen from $Q_{j'}$. Let $T_c'$ be the triangulation occurring immediately before $T_c $ in the algorithm so that $T_c = T_c'(F_i)$. Each maximal simplex of $T_c'$ containing $F_i$ contains $b_{F_{i-1}}$ and hence contains a set $A$ of at least $j'-1$ other vertices from $H'$ (by the inductive hypothesis). Each maximal simplex of $T_c$ containing $b_{F_i}$ either does or does not contain $A\cup \{b_{F_{i-1}}\}$. If it does, then it contains these $j'$ vertices from $H'$ and $|F_i|-1$ vertices from $V(F_i)$, $j-j'-1$ of which are not contained in $H'$. If it does not, then it contains $j'-1$ vertices from $A\cup b_{F_{i-1}}$ and $|F_i|-1$ vertices from $V(F_i)$, $j-j'$ of which are not contained in $H'$. In both cases, there are $j-1$ vertices from $H$ in the maximal simplex containing $b_{F_i}$, and this completes the proof.
\end{proof}

\begin{claim}\label{claim:Qk empty}
    At any iteration in the algorithm, $Q_{j}=\emptyset$ for all $j\geq k+1$.
\end{claim}
\begin{proof}
    Assume that the index $j$ chosen in step (2) is at least $k+1$ in some iteration of the algorithm, and let $S=(F_1,\dots,F_i)$ be the resulting sequence from step (2) as well. By Claim \ref{claim:contains}, each maximal simplex in $T_c$ containing $b_{F_i}$ contains an additional $j-1$ vertices. However, this is impossible since a maximal simplex contains at most $k$ vertices.
\end{proof}

\begin{claim}
The choice of $w,\tau$ in Step (3) of the algorithm is possible as long as the algorithm runs. 
\end{claim}
\begin{proof}
    Let $S=(F_1,\dots,F_i)$ be the sequence obtained from step (2) where $F_i$ was chosen from $Q_j$, and let $H = \bigcup_{F\in S} \left(V(F) \cup \{b_{F}\} \right)$ be the set as in step (3). We will show by induction on $i$ that $r(f(H\setminus \{b_{F_i}\}))\leq j-1$. This will then allow us to prove the claim as we will show. 
    
    For the base case $i=1$, although technically $F_1$ was not chosen from some $Q_j$, we say for the purpose of the inductive step, that $F_1$ was ``chosen'' from $Q_{|F_1|}$. We have that $H= V(F_1) \cup \{b_{F_1}\}$ satisfies $r(f(H\setminus \{b_{F_1}\})) = |F_1|-1$ (since $f(F_1)$ is a circuit). 

    Now, assume the claim holds for $i-1$. Let $S=(F_1,\dots,F_i)$, and let $H = \bigcup_{F\in S} \left(V(F) \cup \{b_{F}\} \right)$ be the set as in step (3). Define $S'=(F_1,\dots,F_{i-1})$ and $H'=\bigcup_{F\in S'} \left(V(F) \cup \{b_{F}\} \right)$. Let $j$ be the index for which $F_i$ was chosen from $Q_j$ in step (2), and let $j'$ be the index for which $F_{i-1}$ was chosen from $Q_{j'}$ (if $i-1=1$, take $j' = |F_1|$). Note that by the definition of $Q_j$, we have that $F_i$ contains $j-j'$ vertices outside of $H'$. 
    By the submodularity of the rank function of a matroid, we have that
    \begin{align*}
    &r(f(H\setminus \{b_{F_i}\}))=r(f(H') \cup f(V(F_i)))\\
    &\leq r(f(H')) + r(f(V(F_i))) - r((f(H') \cap f(V(F_i)))\\
    &\leq j' + |F_i| -1 - (|F_i|- (j-j'))\\
    &= j-1.
    \end{align*}
    Where we have that $r(f(H')) \leq j'$ by the inductive hypothesis, and $r(f(F_i))= |F_i|-1$ since $f(F_i)$ is a circuit. Also, notice that $|V(F_i)\cap H'| = |F_i|- (j-j')$ since $j-j'$ is the number of vertices of $F_i$ outside $H'$. Since $f(V(F_i)\cap H')$ is an independent set, being a proper subset of a circuit, we have that $r(f(V(F_i)\cap H')) = |F_i|- (j-j')$. This concludes the inductive step.
    
    Since $j\leq k$ by Claim \ref{claim:contains}, we conclude that $r(f(H\setminus \{b_{F_i}\})) \leq k-1$. Therefore, taking $G = \cl(f(H\setminus \{b_{F_i}\}))$, there exists some $x\in W\setminus G$ and $\tau \subset \supp(b_{F_i})$ such that $b_{F_i} \in A^x_\tau$. This concludes the proof of the claim.
\end{proof}

\begin{claim}\label{Qj empty}
For any $j\ge 1$, after performing some iteration of the steps (1)-(5) in the algorithm, either $Q_j=\emptyset$ or there is a later iteration of steps (1)-(5) in the algorithm after which  $Q_j = \emptyset$ and the size of $Q_{i}$ for each $1\leq i\leq j-1$ is the same in both iterations.
\end{claim}
\begin{proof}
    We proceed by induction on $k+1-j$. If $k+1-j\leq 0$, then the claim holds by Claim \ref{claim:Qk empty}.

    Let $1\leq j\leq k$.
Assume that after the $i_1$-th iteration of the steps in the algorithm, we have  $Q_j\neq \emptyset$. It suffices to show that there is some later iteration for which the size of $Q_j$ is decreased by $1$ and the size of $Q_i$ for $1\leq i\leq j-1$ remains unchanged. By the induction hypothesis, there exists $i_2\ge i_1$  such that after the steps in the $i_2$-th iteration of the algorithm, $Q_{j+1}=\emptyset$ and the size of $Q_i$ for each $1\leq i\leq j$ is the same as after the $i_1$-th iteration. 
Similarly, we may find $i_3\ge i_2$ such that after the steps in the $i_3$-th iteration of the algorithm, $Q_{j+2}=\emptyset$ and the size of $Q_i$, for every $1\leq i\leq j+1$, is the same as after $i_2$-th iteration (and in particular,  $Q_{j+1}=\emptyset$). Continuing in this way, we find some $i_{k+1-j}\ge i_1$  such that after the steps in the $i_{k+1-j}$-th iteration of the algorithm, $Q_i=\emptyset$ for all $j+1\leq i\leq k$ and the size of $Q_i$ for $1\leq i\leq j$ is the same as after the $i_1$-th iteration. It follows by Claim \ref{claim:Qk empty} that $j$ is the largest index for which $Q_j\neq \emptyset$.
Therefore, in Step (2) of the $i_{k+1-j}+1$-th iteration, we choose some $F\in Q_j$ and set $Q_j=Q_j\setminus \{F\}$, so the size of $Q_j$ decreases by $1$, and the size of $Q_i$ for $1\leq i\leq j-1$ is unchanged. 
\end{proof}

\begin{claim}\label{emptylater}
If after steps (1) - (5) of some iteration of the algorithm $Q_i=\emptyset$ for every $1\leq i\leq j$, then $Q_j=\emptyset$ after the steps in every later iteration of the algorithm.
\end{claim}
\begin{proof}
If $Q_i=\emptyset$ for every $1\leq i\leq j$, then at the next iteration of the algorithm, the index chosen in Step (2) is at least $j+1$. Therefore, it is still the case that $Q_i=\emptyset$ for each $1\leq i\leq j$ in the next iteration. 
\end{proof}

\begin{claim}\label{BadFaces}
After steps (1)-(5) of any iteration of the algorithm, every face of $B(T_c)\setminus B(T)$ is contained in $Q_j$ for some $j$.
\end{claim}
\begin{proof}
We proceed by induction on the number of iterations. The claim is  true in the $0$-th iteration (i.e. the setup), when $T_c=T(F_1)$. 

Let $T_c$ be obtained after the $i$-ith iteration of the algorithm for some $i\geq 0$, and assume the statement holds for $T_c$. Let $T'$ be the triangulation obtained in the $(i+1)$-th iteration, i.e. $T'=T_c(F)$ for some $F$. Let $j$ be the index in Step (2) of the $(i+1)$-th iteration. Then every face of $B(T')\setminus B(T_c)$ contains $b_F$ and hence is placed into $Q_j$ for some $j$ as in step (4). This, combined with the fact that the claim held for $T_c$, implies that the claim holds for $T'$. 
\end{proof}

\begin{claim}\label{terminates}
The algorithm terminates after finitely many steps. The returning triangulation $T'$ satisfies properties (i),(ii) and (iii). 
\end{claim}
\begin{proof}
Applying Claim \ref{Qj empty} for $j=1$, we have that there is some iteration for which $Q_1=\emptyset$. Therefore, by Claim \ref{emptylater}, we have that $Q_1=\emptyset$ at  any later iteration of the algorithm. Now, by Claim \ref{Qj empty}, there is a later iteration in the algorithm for which $Q_1=Q_2=\emptyset$. Applying Claim \ref{emptylater} again, we have that $Q_2=\emptyset$ for every later iteration. Continuing  this way, we get that there is an iteration where $Q_j=\emptyset$ for every $j\geq 1$, and therefore the algorithm  terminates.

By Claim \ref{BadFaces}, every face in  $B(T')\setminus B(T)$ is contained in some $Q_j$. By the above, we have that $Q_j=\emptyset$ for all $j\geq 1$, so $T'$ has no such bad faces. Since $T'$ no longer has the face $F_1$, it follows that $T'$ has strictly less bad faces than $T$. Finally, (i) holds by the definition of $\lambda, f,  y$ throughout the algorithm.
\end{proof}

Claim \ref{terminates} concludes the proof of Theorem \ref{goodtriangulation}.
\end{proof}

\begin{proof}[Proof of Theorem \ref{thm:main}]
We can now prove Theorem \ref{thm:main} by a simple limiting argument. Let $(T_n)_{n\geq 1}$ be a sequence of triangulations of $P$ whose diameters go to 0 with labels $\lambda_n, f_n$, and $y_n$ satisfying properties (P1) - P(3) from Theorem \ref{goodtriangulation}. Such a sequence can be obtained by applying Theorem \ref{goodtriangulation} to sufficiently fine triangulations of $P$. Since $\lambda_n$ is a Sperner-Shapley labeling, there exists a face $\tau_n$ of $T_n$ such that $p\in \conv(\{y_{\lambda_n(v)}^{f_n(v)} \mid v \in V(\tau)\}$, $\{f_n(v) \mid v\in V(\tau_n)\}$ is a basis, and $v\in A_{\lambda_n(v)}^{f_n(v)}$ for all $v \in V(\tau_n)$. By passing to a subsequence, we can assume that the set of triples $\{(f_n(v), y_{\lambda_n(v)}^{f_n(v)}, \lambda_n(v)) \mid v\in V(\tau_n)\}$ is the same for each $n$ and that the sequence of faces $(\tau_n)$ converges to a point $x$. Since the sets are closed, we have that $x\in \bigcap_{v\in V(\tau_1)} A_{\lambda_1(v)}^{f_1(v)}$, which completes the proof.
\end{proof}

\begin{example}\label{exalgorithm}

Consider the triangulation $T$ of a polytope $P$ below with vertices $a,b,c,d$, labeled under $f$ by elements of the realizable rank 3 matroid $\M$ with vectors $V$ below
\[
v_1 = 
\begin{pmatrix}
1\\
1\\
0
\end{pmatrix},\,
v_2 = 
\begin{pmatrix}
-1\\
1\\
0
\end{pmatrix},\,
v_3 = 
\begin{pmatrix}
0\\
1\\
0
\end{pmatrix},\,
v_4 = 
\begin{pmatrix}
0\\
0\\
1
\end{pmatrix}.
\] 
Assume we have an $\M$-Komiya cover of $P$.
In Figure 1A, we see that the only bad face is the edge $\{a,b\}$, so we take $F_1 = \{a,b\}$ in the setup of the algorithm and we set $T_c=T(F_1)$. We then label $b_{F_1}$ under $f$ by an element from $w\in V\setminus \{v_1\}$ such that $b_{F_1} \in A^w_\sigma$ for some $\sigma$; here we labeled it by $v_3$ as in Figure 1B. We then have $S=({F_1})$ and 
\[
Q_3 = \{  \{b_{F_1},a,c\},\, \{b_{F_1},b,c \},\, \{b_{F_1},d\}\}.
\]
In the first iteration of the algorithm, we choose $F = \{b_{F_1},a,c\} \in Q_3$ and remove it from $Q_3$, set $T_c = T_c(F)$, and update $S=({F_1}, F)$. Then we choose an element in $V\setminus \cl(\{v_1,v_3\})$ to label $B_F$ under $f$ for which there is only one choice, $v_1$. In step (4), $B(T_c; b_F) = \emptyset$, so we remove $F$ from $S$, and the $Q_j$'s remain unchanged.

In the second iteration, we have $Q_3 = \{\{b_{F_1},b,c \},\, \{b_{F_1},d\}\}$, and we choose $\{b_{F_1},b,c \}$. The remaining steps of this iteration are then carried out identically to the first iteration, and we arrive at the triangulation in Figure 1C.

For the third iteration, we choose the remaining face $\{b_{F_1},d\}$ from $Q_3$, remove it from $Q_3$ and update $T_c=T_c(\{b_{F_1},d\})$. We choose an element $w\in V\setminus \cl(\{v_1,v_3\})$ such that $b_{\{b_{F_1},d\}}\in A^w_\sigma$ for some $\sigma$, for which $v_1$ is the only choice. In step (4), $B(T_c; b_{\{b_{F_1},d\}})=\emptyset$, so the $Q_j$'s remain unchanged. Finally, we return to step (1) and since $Q_j= \emptyset $ for all $j$, we terminate the algorithm and we are left with the triangulation in Figure 1D.
\begin{figure}[h]\label{fig:main}
    \centering
 \begin{subfigure}{0.4\textwidth}
         \centering
    \begin{tikzpicture}[scale=.6]
    \coordinate (a) at (-2,0);
    \coordinate (b) at (2,0);
    \coordinate (c) at (0,3.464);
    \coordinate (d) at (0,-3.464);
     \fill[black, draw=black, thick] (a) circle (1.5pt) node[black, below left] {$a,v_1$};
    \fill[black, draw=black, thick] (b) circle (1.5pt) node[black, below right] {$b,v_1$};
    \fill[black, draw=black, thick] (c) circle (1.5pt) node[black, above] {$c,v_2$};
    \fill[black, draw=black, thick] (d) circle (1.5pt) node[black, below] {$d,v_3$};
    \draw (a) -- (b);
    \draw (a) -- (c);
    \draw (c) -- (b);
    \draw (a) -- (d);
    \draw (d) -- (b);
    \end{tikzpicture}
        \caption{The given triangulation and $f$ labels. The edge $\{a,b\}$ is bad.}\label{1A}
     \end{subfigure}
     \hfill
         \begin{subfigure}{0.3\textwidth}
         \centering
    \begin{tikzpicture}[scale=.6]
    \coordinate (a) at (-2,0);
    \coordinate (b) at (2,0);
    \coordinate (c) at (0,3.464);
    \coordinate (d) at (0,-3.464);
    \coordinate (e) at (0,0);

    \fill[black, draw=black, thick] (e) circle (1.5pt) node[black, below right] {$v_3$};
    \fill[black, draw=black, thick] (a) circle (1.5pt) node[black, below left] {$a,v_1$};
    \fill[black, draw=black, thick] (b) circle (1.5pt) node[black, below right] {$b,v_1$};
    \fill[black, draw=black, thick] (c) circle (1.5pt) node[black, above] {$c,v_2$};
    \fill[black, draw=black, thick] (d) circle (1.5pt) node[black, below] {$d,v_3$};
    \draw (a) -- (b);
    \draw (a) -- (c);
    \draw (c) -- (b);
    \draw (a) -- (d);
    \draw (d) -- (b);

    \draw (c) -- (e);
    \draw (d) -- (e);
    \end{tikzpicture}
     \caption{The current triangulation and $f$ labels after the setup of the algorithm.}\label{1B}
     \end{subfigure}
         \begin{subfigure}{0.3\textwidth}
    \begin{tikzpicture}[scale=.6]
    \coordinate (a) at (-2,0);
    \coordinate (b) at (2,0);
    \coordinate (c) at (0,3.464);
    \coordinate (d) at (0,-3.464);
    \coordinate (e) at (0,0);
    \coordinate (f) at (-2/3,3.464/3);
    \coordinate (g) at (2/3,3.464/3);

    \fill[black, draw=black, thick] (a) circle (1.5pt) node[black, below left] {$a,v_1$};
    \fill[black, draw=black, thick] (b) circle (1.5pt) node[black, below right] {$b,v_1$};
    \fill[black, draw=black, thick] (c) circle (1.5pt) node[black, above] {$c,v_2$};
    \fill[black, draw=black, thick] (d) circle (1.5pt) node[black, below] {$d,v_3$};
    \fill[black, draw=black, thick] (e) circle (1.5pt) node[black, below right] {$v_3$};
    \fill[black, draw=black, thick] (f) circle (1.5pt) node[black,below] {};
    \fill[black, draw=black, thick] (-2/3-.1,3.464/3 - 0.2) circle (0pt) node[black, below] {$v_4$};
    \fill[black, draw=black, thick] (g) circle (1.5pt) node[black,below] {};
    \fill[black, draw=black, thick] (2/3+.2,3.464/3 - 0.2) circle (0pt) node[black, below] {$v_4$};
    \draw (a) -- (b);
    \draw (a) -- (c);
    \draw (c) -- (b);
    \draw (a) -- (d);
    \draw (d) -- (b);

    \draw (c) -- (e);
    \draw (d) -- (e);

    \draw (f) -- (a);
    \draw (f) -- (e);
    \draw (f) -- (c);

    \draw (g) -- (b);
    \draw (g) -- (e);
    \draw (g) -- (c);
    \end{tikzpicture}
    \caption{The current triangulation and $f$ labels after the first and second iteration of the algorithm.}\label{1C}
     \end{subfigure}
     \hfill
         \begin{subfigure}{0.3\textwidth}
         \centering
    \begin{tikzpicture}[scale=.6]
    \coordinate (a) at (-2,0);
    \coordinate (b) at (2,0);
    \coordinate (c) at (0,3.464);
    \coordinate (d) at (0,-3.464);
    \coordinate (e) at (0,0);
    \coordinate (f) at (-2/3,3.464/3);
    \coordinate (g) at (2/3,3.464/3);
    \coordinate (h) at (0,-3.464/2);

    \fill[black, draw=black, thick] (a) circle (1.5pt) node[black, below left] {$a,v_1$};
    \fill[black, draw=black, thick] (b) circle (1.5pt) node[black, below right] {$b,v_1$};
    \fill[black, draw=black, thick] (c) circle (1.5pt) node[black, above] {$c,v_2$};
    \fill[black, draw=black, thick] (d) circle (1.5pt) node[black, below] {$d,v_3$};
    \fill[black, draw=black, thick] (e) circle (1.5pt) node[black, below right] {$v_3$};
    \fill[black, draw=black, thick] (f) circle (1.5pt) node[black,below] {};
    \fill[black, draw=black, thick] (-2/3-.1,3.464/3 - 0.2) circle (0pt) node[black, below] {$v_4$};
    \fill[black, draw=black, thick] (g) circle (1.5pt) node[black,below] {};
    \fill[black, draw=black, thick] (2/3+.2,3.464/3 - 0.2) circle (0pt) node[black, below] {$v_4$};
    \fill[black, draw=black, thick] (h) circle (1.5pt) node[black,below left] {$v_4$};
    
    \draw (a) -- (b);
    \draw (a) -- (c);
    \draw (c) -- (b);
    \draw (a) -- (d);
    \draw (d) -- (b);

    \draw (c) -- (e);
    \draw (d) -- (e);

    \draw (f) -- (a);
    \draw (f) -- (e);
    \draw (f) -- (c);

    \draw (g) -- (b);
    \draw (g) -- (e);
    \draw (g) -- (c);

    \draw (h) -- (a);
    \draw (h) -- (b);
    \end{tikzpicture}
    \caption{The current triangulation and $f$ values after the third and final iteration of the algorithm.}
         \end{subfigure}
  
\caption{Demonstration of the algorithm from the proof of Theorem \ref{goodtriangulation}.}
\label{fig:ex}
\end{figure}
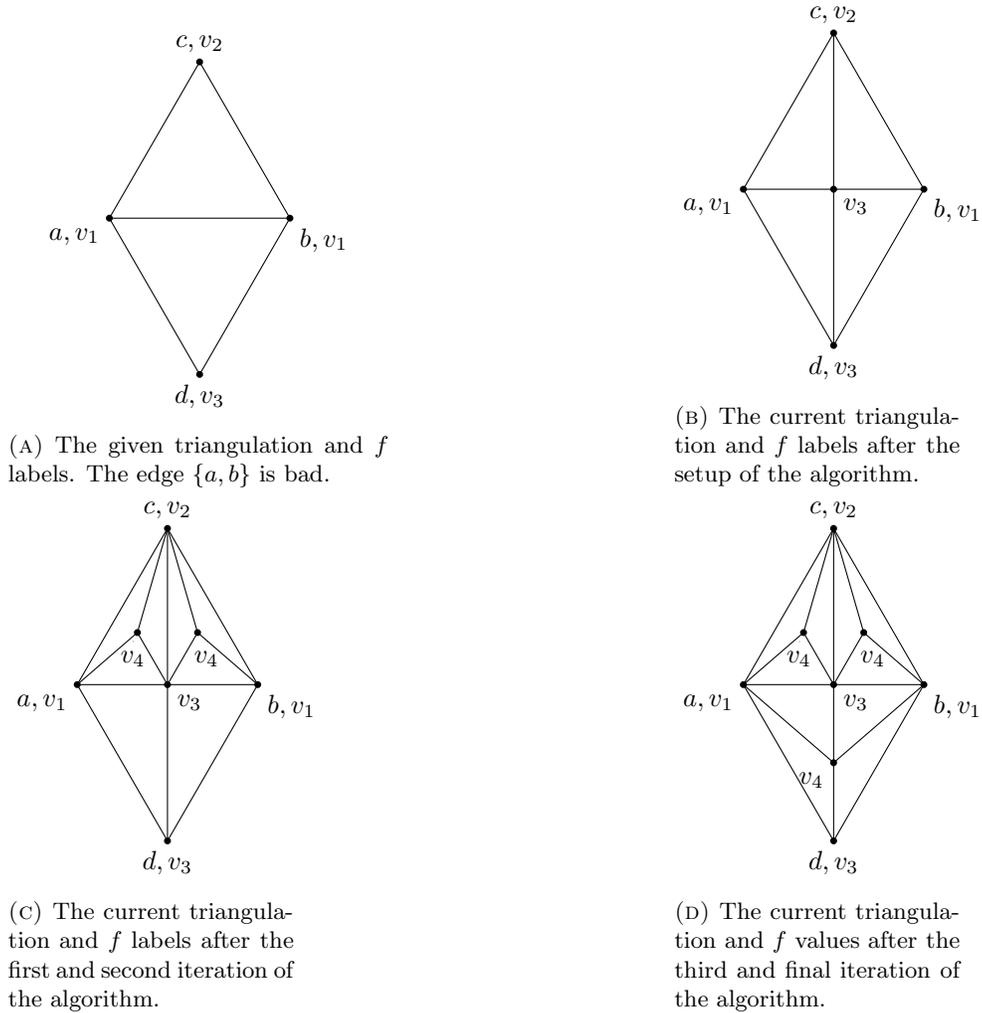

\end{example}

\end{document}